\def\bu{\bullet}
\def\marker{\>\hbox{${\vcenter{\vbox{
    \hrule height 0.4pt\hbox{\vrule width 0.4pt height 6pt
    \kern6pt\vrule width 0.4pt}\hrule height 0.4pt}}}$}\>}
\def\gpic#1{#1
     \smallskip\par\noindent{\centerline{\box\graph}} \medskip}
\documentclass[12pt]{article}
\usepackage{amsmath,amssymb,amsthm}
\usepackage{graphicx}
\usepackage{subfigure}
\usepackage{psfrag}
\usepackage{color}
\usepackage{enumerate}
\usepackage{epstopdf}
\newtheorem{thm}{Theorem}[section]

\newtheorem{lem}[thm]{Lemma}
\newtheorem{prop}[thm]{Proposition}
\theoremstyle{definition}
\newtheorem{defn}[thm]{Definition}

\def\dfn#1{{\it #1}}

\def\less{-}

\def\cF{\mathcal{F}}
\def\sat{{\rm sat}}
\def\ex{{\rm ex}}
\def\NN{\mathbb{N}}
\def\Gb{\overline{G}}
\def\Kb{\overline{K}}
\def\FL#1{\left\lfloor #1\right\rfloor}
\def\CH#1#2{\binom{#1}{#2}}
\def\FR#1#2{\frac{#1}{#2}}
\def\VEC#1#2#3{#1_{#2},\ldots,#1_{#3}}
\def\VECOP#1#2#3#4{#1_{#2}#4\cdots#4#1_{#3}}

\def\Sb{\overline{S}}
\def\Tb{\overline{T}}
\def\esub{\subseteq}

\title{Extremal problems on saturation for the family of $k$-edge-connected
graphs}

\author{
Hui Lei\footnotemark[1],\quad
Suil O\footnotemark[2],\quad
Yongtang Shi\footnotemark[1],\quad
Douglas B. West\footnotemark[3],\quad
Xuding Zhu\footnotemark[4]
}
\date{\today}

\begin{document}
\maketitle
\begin{abstract}
Let $\cF$ be a family of graphs.  A graph $G$ is \dfn{$\mathcal{F}$-saturated}
if $G$ contains no member of $\cF$ as a subgraph but $G+e$ contains some
member of $\cF$ whenever $e\in E(\overline{G})$.  The \dfn{saturation number}
and \dfn{extremal number} of $\mathcal{F}$, denoted $\sat(n,\mathcal{F})$ and
$\ex(n,\mathcal{F})$ respectively, are the minimum and maximum numbers of edges
among $n$-vertex $\mathcal{F}$-saturated graphs.  For $k\in\NN$, let
$\mathcal{F}_k$ and $\mathcal{F}'_k$ be the families of $k$-connected and
$k$-edge-connected graphs, respectively.  Wenger proved
$\sat(n,\mathcal{F}_k)=(k-1)n-{k\choose2}$; we prove
$\sat(n,\mathcal{F}'_k)=(k-1)(n-1)-\FL{\frac {n}{k+1}}\CH{k-1}2$.
We also prove $\ex(n,\cF'_k)=(k-1)n-{k\choose2}$ and characterize when
equality holds.  Finally, we give a lower bound on the spectral radius for
$\mathcal{F}_k$-saturated and $\mathcal{F}'_k$-saturated graphs.
\\

\end{abstract}

\renewcommand{\thefootnote}{\fnsymbol{footnote}}
\footnotetext[1]{
Center for Combinatorics and LPMC,
Nankai University, Tianjin 300071, China;
leihui0711@163.com, shi@nankai.edu.cn.  Research supported by the National
Natural Science Foundation of China and the Natural Science Foundation of
Tianjin No.17JCQNJC00300.}
\footnotetext[2]{
Department of Applied Mathematics and Statistics,
The State University of New York, Korea, Incheon, 21985;
suil.o@sunykorea.ac.kr (corresponding author).
Research supported by NRF-2017R1D1A1B03031758.}
\footnotetext[3]{
Departments of Mathematics, Zhejiang Normal University, Jinhua, 321004
and University of Illinois, Urbana, IL 61801, USA; dwest@math.uiuc.edu.
Research supported by Recruitment Program of Foreign Experts, 1000
Talent Plan, State Administration of Foreign Experts Affairs, China.}
\footnotetext[4]{
Department of Mathematics, Zhejiang Normal University, Jinhua, 321004;
xdzhu@zjnu.edu.cn.  Research supported in part by CNSF 00571319.}

\section{Introduction}

\baselineskip 17pt

When $\cF$ is a family of graphs, a graph $G$ is \dfn{$\mathcal{F}$-saturated}
if (1) no subgraph of $G$ belongs to $\cF$, and (2) adding to $G$ any edge of
its complement $\Gb$ completes a subgraph that belongs to $\cF$ (our definition
of ``graph'' prohibits loops and multiedges).
The \dfn{saturation number} of $\mathcal{F}$, denoted $\sat(n,\mathcal{F})$, is
the least number of edges in an $n$-vertex $\mathcal{F}$-saturated graph.
The \dfn{extremal number} $\ex(n,\cF)$ is the maximum number of edges in
an $n$-vertex $\cF$-saturated graph.  When $\cF$ has only one graph $F$, we
simply write $\sat(n,F)$ and $\ex(n,F)$, such as when $F$ is $K_t$, the
complete graph with $t$ vertices.

Initiating the study of extremal graph theory, Tur\'an~\cite{Tur} determined
the extremal number $\ex(n,K_{r+1})$; the unique extremal graph is the
$n$-vertex complete $r$-partite graph whose part-sizes differ by at most $1$.
Saturation numbers were first studied by Erd\H{o}s, Hajnal, and Moon
\cite{EHM1964}; they proved $\sat(n,K_{k+1})=(k-1)n-{{k}\choose{2}}$.  They
also proved that equality holds only for the graph formed from a copy of
$K_{k-1}$ with vertex set $S$ by adding $n-k+1$ vertices that each have
neighborhood $S$.  We call this the {\it complete split graph} $S_{n,k}$;
note that $S_{n,k}$ has clique number $k$ and no $k$-connected subgraph,
and $S_{n,2}$ is a star.  For an excellent survey on saturation numbers, we
refer the reader to Faudree, Faudree, and Schmitt \cite{FFS2011}.  

In this paper, we study the relationship between saturation and 
edge-connectivity.  For a given positive integer $k$, let
$\mathcal{F}_k$ be the family of $k$-connected graphs, and let
$\mathcal{F}'_k$ be the family of $k$-edge-connected graphs.
Wenger \cite{W2014} determined $\sat(n,\mathcal{F}_k)$.
Since $K_{k+1}$ is a minimal $k$-connected graph, it is not surprising
that $S_{n,k}$ is also a smallest $\cF_k$-saturated graph, but in fact
the family of extremal graphs is much larger.  A \dfn{$k$-tree} is any
graph obtained from $K_k$ by iteratively introducing a new vertex whose
neighborhood in the previous graph consists of $k$ pairwise adjacent vertices.
Note that $S_{n,k}$ is a $(k-1)$-tree.

\begin{thm}[{\rm Wenger \cite{W2014}}]\label{k-con}
$\sat(n,\mathcal{F}_k)=(k-1)n-{k\choose2}$ when $n\geq k$.  Furthermore, every
$(k-1)$-tree with $n$ vertices has this many edges and is $\cF_k$-saturated.
\end{thm}

For $n\ge k+1$, we determine $\sat(\cF'_k)$ and $\ex(\cF'_k)$.
An $\cF'_1$-saturated graph has no edges, so henceforth we may assume $k\ge2$.
Let $\rho_k(n)=(k-1)(n-1)-\FL{\FR{n}{k+1}}\CH{k-1}2$.  In Section 2, we
construct for $n\ge k+1$ an $\mathcal{F}'_k$-saturated graph with $n$ vertices
having $\rho_k(n)$ edges, proving $\sat(n,\cF'_k)\le\rho_k(n)$.  Using
induction on $n$, in Section 3 we prove that if $G$ is
$\mathcal{F}_k'$-saturated, then $\rho_k(n)\le |E(G)|\le(k-1)n-\CH{k}{2}$,
where $E(G)$ denotes the edge set of a graph $G$.  Since $S_{n,k}$ is also
$\cF'_k$-saturated, the upper bound is sharp.  Thus
$\sat(n,\mathcal{F}'_k)=\rho_k(n)$ and $\ex(n,\mathcal{F}'_k)=(k-1)n-\CH k2$.

The spectral radius of a graph is the largest eigenvalue of its adjacency
matrix.  In Section 4, we give a lower bound on the spectral radius for
$\mathcal{F}_k$-saturated and $\mathcal{F}'_k$-saturated graphs,

Additional notation is as follows.  For $v\in V(G)$, let $d_G(v)$ and $N_G(v)$
denote the degree and the neighborhood of $v$ in $G$, respectively.  For 
$A,B\subseteq V(G)$, let $\overline{A}=V(G)\less A$, let $[A,B]$ be the set of
edges with endpoints in $A$ and $B$, and let $G[A]$ to denote the subgraph
of $G$ induced by $A$.  Let $[k]=\{1,2,\ldots,k\}$.

Let $K_{k+1}^-$ denote the graph obtained from $K_{k+1}$ by deleting one edge;
this graph is the unique smallest $\cF'_k$-saturated graph that is not a
complete graph.  The complete graphs with at most $k$ vertices are trivially 
$\cF'_k$-saturated, since there are no edges to add.  We therefore use
{\it nontrivial $\cF'_k$-saturated graph} to mean an $\cF'_k$-saturated graph
with at least $k+1$ vertices.

\section{Construction}

Recall that $\rho_k(n)=(k-1)(n-1)-\FL{\FR{n}{k+1}}\CH{k-1}2$ and that we
restrict to $k\ge2$ since $\cF'_1$-saturated graphs have no edges.  In this
section, for $n\ge k+1$, we construct an $n$-vertex $\mathcal{F}'_k$-saturated
graph with $\rho_k(n)$ edges.  Since every $\cF'_2$-saturated graph is a tree
(and $\rho_2(n)=n-1$), we need only consider $k\ge3$.

\begin{defn}\label{Gkn}
Fix $k\in\NN$ with $k\ge3$.  For $n\in\NN$ with $n>k$,
let $t=\FL{\FR n{k+1}}$ and $r=n-t(k+1)$.  Let $H_{i}$ be a copy of $K_{k+1}^-$
using vertices $u_{i,1},\dots,u_{i,k+1}$, with $u_{i,1}$ and $u_{i,k+1}$
nonadjacent.  Let $U_i=V(H_i)$ for $i\in[t]$.
Let $F_{t}$ be the graph obtained from the disjoint union
$H_{1}+\cdots+H_{t}$ by adding the edge $u_{i,j}u_{i+1,j}$ for all $i$ and
$j$ such that $i\in[t-1]$ and $j\in[k+1]\less\{2,k\}$.  Let $G_{k,n}$
be the graph obtained from $F_{t}$ by adding new vertices $\VEC w1r$, each
having neighborhood $V(H_{t})\less\{u_{t,1},u_{t,k+1}\}$.
\end{defn}

\begin{figure}[h]
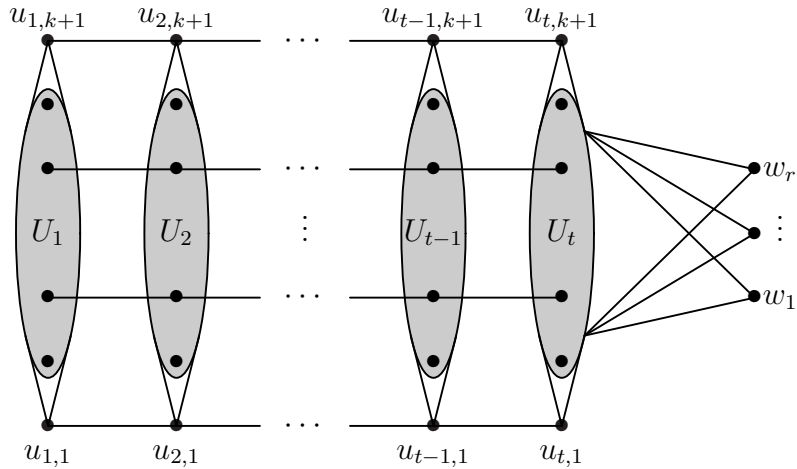

\gpic{
\expandafter\ifx\csname graph\endcsname\relax \csname newbox\endcsname\graph\fi
\expandafter\ifx\csname graphtemp\endcsname\relax \csname newdimen\endcsname\graphtemp\fi
\setbox\graph=\vtop{\vskip 0pt\hbox{%
    \graphtemp=.5ex\advance\graphtemp by 2.151in
    \rlap{\kern 0.168in\lower\graphtemp\hbox to 0pt{\hss $\bu$\hss}}%
    \graphtemp=.5ex\advance\graphtemp by 2.151in
    \rlap{\kern 0.840in\lower\graphtemp\hbox to 0pt{\hss $\bu$\hss}}%
    \graphtemp=.5ex\advance\graphtemp by 2.151in
    \rlap{\kern 2.185in\lower\graphtemp\hbox to 0pt{\hss $\bu$\hss}}%
    \graphtemp=.5ex\advance\graphtemp by 2.151in
    \rlap{\kern 2.857in\lower\graphtemp\hbox to 0pt{\hss $\bu$\hss}}%
    \graphtemp=.5ex\advance\graphtemp by 0.134in
    \rlap{\kern 0.168in\lower\graphtemp\hbox to 0pt{\hss $\bu$\hss}}%
    \graphtemp=.5ex\advance\graphtemp by 0.134in
    \rlap{\kern 0.840in\lower\graphtemp\hbox to 0pt{\hss $\bu$\hss}}%
    \graphtemp=.5ex\advance\graphtemp by 0.134in
    \rlap{\kern 2.185in\lower\graphtemp\hbox to 0pt{\hss $\bu$\hss}}%
    \graphtemp=.5ex\advance\graphtemp by 0.134in
    \rlap{\kern 2.857in\lower\graphtemp\hbox to 0pt{\hss $\bu$\hss}}%
    \special{pn 11}%
    \special{sh 0.200}%
    \special{ar 168 1143 168 756 0 6.28319}%
    \graphtemp=.5ex\advance\graphtemp by 1.143in
    \rlap{\kern 0.168in\lower\graphtemp\hbox to 0pt{\hss $U_1$\hss}}%
    \graphtemp=.5ex\advance\graphtemp by 2.286in
    \rlap{\kern 0.168in\lower\graphtemp\hbox to 0pt{\hss $u_{1,1}$\hss}}%
    \graphtemp=.5ex\advance\graphtemp by 0.000in
    \rlap{\kern 0.168in\lower\graphtemp\hbox to 0pt{\hss $u_{1,k+1}$\hss}}%
    \graphtemp=.5ex\advance\graphtemp by 1.815in
    \rlap{\kern 0.168in\lower\graphtemp\hbox to 0pt{\hss $\bu$\hss}}%
    \graphtemp=.5ex\advance\graphtemp by 1.479in
    \rlap{\kern 0.168in\lower\graphtemp\hbox to 0pt{\hss $\bu$\hss}}%
    \graphtemp=.5ex\advance\graphtemp by 0.807in
    \rlap{\kern 0.168in\lower\graphtemp\hbox to 0pt{\hss $\bu$\hss}}%
    \graphtemp=.5ex\advance\graphtemp by 0.471in
    \rlap{\kern 0.168in\lower\graphtemp\hbox to 0pt{\hss $\bu$\hss}}%
    \special{pa 49 608}%
    \special{pa 168 134}%
    \special{fp}%
    \special{pa 168 134}%
    \special{pa 287 608}%
    \special{fp}%
    \special{pa 49 1678}%
    \special{pa 168 2151}%
    \special{fp}%
    \special{pa 168 2151}%
    \special{pa 287 1678}%
    \special{fp}%
    \special{sh 0.200}%
    \special{ar 840 1143 168 756 0 6.28319}%
    \graphtemp=.5ex\advance\graphtemp by 1.143in
    \rlap{\kern 0.840in\lower\graphtemp\hbox to 0pt{\hss $U_2$\hss}}%
    \graphtemp=.5ex\advance\graphtemp by 2.286in
    \rlap{\kern 0.840in\lower\graphtemp\hbox to 0pt{\hss $u_{2,1}$\hss}}%
    \graphtemp=.5ex\advance\graphtemp by 0.000in
    \rlap{\kern 0.840in\lower\graphtemp\hbox to 0pt{\hss $u_{2,k+1}$\hss}}%
    \graphtemp=.5ex\advance\graphtemp by 1.815in
    \rlap{\kern 0.840in\lower\graphtemp\hbox to 0pt{\hss $\bu$\hss}}%
    \graphtemp=.5ex\advance\graphtemp by 1.479in
    \rlap{\kern 0.840in\lower\graphtemp\hbox to 0pt{\hss $\bu$\hss}}%
    \graphtemp=.5ex\advance\graphtemp by 0.807in
    \rlap{\kern 0.840in\lower\graphtemp\hbox to 0pt{\hss $\bu$\hss}}%
    \graphtemp=.5ex\advance\graphtemp by 0.471in
    \rlap{\kern 0.840in\lower\graphtemp\hbox to 0pt{\hss $\bu$\hss}}%
    \special{pa 721 608}%
    \special{pa 840 134}%
    \special{fp}%
    \special{pa 840 134}%
    \special{pa 959 608}%
    \special{fp}%
    \special{pa 721 1678}%
    \special{pa 840 2151}%
    \special{fp}%
    \special{pa 840 2151}%
    \special{pa 959 1678}%
    \special{fp}%
    \special{pa 168 2151}%
    \special{pa 1277 2151}%
    \special{fp}%
    \special{pa 168 1479}%
    \special{pa 1277 1479}%
    \special{fp}%
    \special{pa 168 807}%
    \special{pa 1277 807}%
    \special{fp}%
    \special{pa 168 134}%
    \special{pa 1277 134}%
    \special{fp}%
    \special{sh 0.200}%
    \special{ar 2185 1143 168 756 0 6.28319}%
    \graphtemp=.5ex\advance\graphtemp by 1.143in
    \rlap{\kern 2.185in\lower\graphtemp\hbox to 0pt{\hss $U_{t-1}$\hss}}%
    \graphtemp=.5ex\advance\graphtemp by 2.286in
    \rlap{\kern 2.185in\lower\graphtemp\hbox to 0pt{\hss $u_{t-1,1}$\hss}}%
    \graphtemp=.5ex\advance\graphtemp by 0.000in
    \rlap{\kern 2.185in\lower\graphtemp\hbox to 0pt{\hss $u_{t-1,k+1}$\hss}}%
    \graphtemp=.5ex\advance\graphtemp by 1.815in
    \rlap{\kern 2.185in\lower\graphtemp\hbox to 0pt{\hss $\bu$\hss}}%
    \graphtemp=.5ex\advance\graphtemp by 1.479in
    \rlap{\kern 2.185in\lower\graphtemp\hbox to 0pt{\hss $\bu$\hss}}%
    \graphtemp=.5ex\advance\graphtemp by 0.807in
    \rlap{\kern 2.185in\lower\graphtemp\hbox to 0pt{\hss $\bu$\hss}}%
    \graphtemp=.5ex\advance\graphtemp by 0.471in
    \rlap{\kern 2.185in\lower\graphtemp\hbox to 0pt{\hss $\bu$\hss}}%
    \special{pa 2066 608}%
    \special{pa 2185 134}%
    \special{fp}%
    \special{pa 2185 134}%
    \special{pa 2304 608}%
    \special{fp}%
    \special{pa 2066 1678}%
    \special{pa 2185 2151}%
    \special{fp}%
    \special{pa 2185 2151}%
    \special{pa 2304 1678}%
    \special{fp}%
    \special{sh 0.200}%
    \special{ar 2857 1143 168 756 0 6.28319}%
    \graphtemp=.5ex\advance\graphtemp by 1.143in
    \rlap{\kern 2.857in\lower\graphtemp\hbox to 0pt{\hss $U_t$\hss}}%
    \graphtemp=.5ex\advance\graphtemp by 2.286in
    \rlap{\kern 2.857in\lower\graphtemp\hbox to 0pt{\hss $u_{t,1}$\hss}}%
    \graphtemp=.5ex\advance\graphtemp by 0.000in
    \rlap{\kern 2.857in\lower\graphtemp\hbox to 0pt{\hss $u_{t,k+1}$\hss}}%
    \graphtemp=.5ex\advance\graphtemp by 1.815in
    \rlap{\kern 2.857in\lower\graphtemp\hbox to 0pt{\hss $\bu$\hss}}%
    \graphtemp=.5ex\advance\graphtemp by 1.479in
    \rlap{\kern 2.857in\lower\graphtemp\hbox to 0pt{\hss $\bu$\hss}}%
    \graphtemp=.5ex\advance\graphtemp by 0.807in
    \rlap{\kern 2.857in\lower\graphtemp\hbox to 0pt{\hss $\bu$\hss}}%
    \graphtemp=.5ex\advance\graphtemp by 0.471in
    \rlap{\kern 2.857in\lower\graphtemp\hbox to 0pt{\hss $\bu$\hss}}%
    \special{pa 2738 608}%
    \special{pa 2857 134}%
    \special{fp}%
    \special{pa 2857 134}%
    \special{pa 2976 608}%
    \special{fp}%
    \special{pa 2738 1678}%
    \special{pa 2857 2151}%
    \special{fp}%
    \special{pa 2857 2151}%
    \special{pa 2976 1678}%
    \special{fp}%
    \special{pa 2857 2151}%
    \special{pa 1748 2151}%
    \special{fp}%
    \special{pa 2857 1479}%
    \special{pa 1748 1479}%
    \special{fp}%
    \special{pa 2857 807}%
    \special{pa 1748 807}%
    \special{fp}%
    \special{pa 2857 134}%
    \special{pa 1748 134}%
    \special{fp}%
    \graphtemp=.5ex\advance\graphtemp by 2.151in
    \rlap{\kern 1.513in\lower\graphtemp\hbox to 0pt{\hss $\cdots$\hss}}%
    \graphtemp=.5ex\advance\graphtemp by 1.479in
    \rlap{\kern 1.513in\lower\graphtemp\hbox to 0pt{\hss $\cdots$\hss}}%
    \graphtemp=.5ex\advance\graphtemp by 1.143in
    \rlap{\kern 1.513in\lower\graphtemp\hbox to 0pt{\hss $\vdots$\hss}}%
    \graphtemp=.5ex\advance\graphtemp by 0.807in
    \rlap{\kern 1.513in\lower\graphtemp\hbox to 0pt{\hss $\cdots$\hss}}%
    \graphtemp=.5ex\advance\graphtemp by 0.134in
    \rlap{\kern 1.513in\lower\graphtemp\hbox to 0pt{\hss $\cdots$\hss}}%
    \graphtemp=.5ex\advance\graphtemp by 1.479in
    \rlap{\kern 3.866in\lower\graphtemp\hbox to 0pt{\hss $\bu$\hss}}%
    \graphtemp=.5ex\advance\graphtemp by 1.143in
    \rlap{\kern 3.866in\lower\graphtemp\hbox to 0pt{\hss $\bu$\hss}}%
    \graphtemp=.5ex\advance\graphtemp by 0.807in
    \rlap{\kern 3.866in\lower\graphtemp\hbox to 0pt{\hss $\bu$\hss}}%
    \graphtemp=.5ex\advance\graphtemp by 1.479in
    \rlap{\kern 4.000in\lower\graphtemp\hbox to 0pt{\hss $w_1$\hss}}%
    \graphtemp=.5ex\advance\graphtemp by 1.143in
    \rlap{\kern 4.000in\lower\graphtemp\hbox to 0pt{\hss $\vdots$\hss}}%
    \graphtemp=.5ex\advance\graphtemp by 0.807in
    \rlap{\kern 4.000in\lower\graphtemp\hbox to 0pt{\hss $w_r$\hss}}%
    \special{pa 2976 608}%
    \special{pa 3866 1479}%
    \special{fp}%
    \special{pa 3866 1479}%
    \special{pa 2976 1678}%
    \special{fp}%
    \special{pa 2976 608}%
    \special{pa 3866 1143}%
    \special{fp}%
    \special{pa 3866 1143}%
    \special{pa 2976 1678}%
    \special{fp}%
    \special{pa 2976 608}%
    \special{pa 3866 807}%
    \special{fp}%
    \special{pa 3866 807}%
    \special{pa 2976 1678}%
    \special{fp}%
    \hbox{\vrule depth2.286in width0pt height 0pt}%
    \kern 4.000in
  }%
}%
}

\vspace{-1pc}
\caption{The graph $G_{k,n}$.\label{gkn}}
\end{figure}

\begin{prop}\label{construction}
For $n>k\ge3$, the graph $G_{k,n}$ of Definition~\ref{Gkn} is
$\mathcal{F}'_k$-saturated and has $n$ vertices and $\rho_k(n)$ edges.
\end{prop}

\begin{proof}
Since $n=t(k+1)+r$, the graph $G_{k,n}$ has $n$ vertices.

In $G_{k,n}$, the vertices $\VEC w1r$ have degree $k-1$ and hence cannot lie in
a $k$-edge-connected subraph.  In $F_{t}$, the edges joining $U_i$ and
$U_{i+1}$ form a cut of size $k-1$, so any $k$-edge-connected subgraph of
$G_{k,n}$ is contained in just one copy of $K_{k+1}^-$.  However, $K_{k+1}^-$
has two vertices of degree $k-1$, leaving only $k-1$ other vertices.  Hence
$G_{k,n}$ has no $k$-edge-connected subgraph.

In $F_{t}$, there are $t\left[\CH{k+1}2-1\right]+(k-1)(t-1)$ edges.
The added vertices $\VEC w1r$ contribute $r(k-1)$ more edges.  Since
$n=t(k+1)+r$, we compute
\begin{align*}
|E(G_{k,n})|
&=t\left[\CH{k+1}2-1\right]+(k-1)(t+r-1)= t\FR{k^2+3k-4}2+(k-1)(r-1)\\
&=t\FR{(k-1)(k+4)}2+(k-1)(r-1)=(k-1)[t(k+1)+r-1]-t\CH{k-1}2\\
&=(k-1)(n-1)-t\CH{k-1}2=\rho_k(n).\\
\end{align*}

\vspace{-1pc}
Let $xy$ be an edge in the complement of $G_{k,n}$.
It remains to show that the graph $G'$ obtained by adding $xy$ to $G_{k,n}$
has a $k$-edge-connected subgraph.  Note that the subgraph of $G_{k,n}$
induced by $U_t\cup\{\VEC w1r\}$ is the $K_{k+1}$-saturated graph $S_{k+r+1,k}$
of~\cite{EHM1964}, so $G'$ contains $K_{k+1}$ when $x$ and $y$ lie
in this set.  Similarly, if $xy$ is the one missing edge of $H_{i}$, then
$G'$ again contains $K_{k+1}$.  Hence we may assume that 
$x\in U_i$ with $1\le i<t$ and that $y\in\{\VEC w1r\}$ or $y\in U_j$ with
$i<j\le t$.  If $y\in\{\VEC w1r\}$, then let $j=t+1$ and $U_j=\{y\}$,
in order to combine cases.
Let $H'$ be the subgraph of $G'$ induced by $\bigcup_{l=i}^{j} U_l$.
To prove that $H'$ is $k$-edge-connected, we show that $H'-S$ is connected,
where $S$ is a set of $k-1$ edges in $H'$.

Suppose first that $H'[U_l]-S$ is disconnected for some $l$ with $i\le l\le j$
(this can only occur with $l\le t$).  Since $\kappa'(H_{l})=k-1$ for
$l \in [t]$, this case requires $S\esub E(H'[U_l])$.  In $H'-S$, every vertex
of $U_l$ except $u_{l,2}$ and $u_{l,k}$ has a neighbor
in $U_{l-1}$ when $l>i$ and in $U_{l+1}$ when $l<j$.  Also $u_{l,2}$ and 
$u_{l,k}$ have degree $k$ in $H'$, so in $H'-S$ each has a neighbor in $U_l$.
If one of them is the only neighbor of the other in $H'-S$, then in $H'-S$
it has an additional neighbor in $U_l$.  Thus in $H'-S$ each component of the
subgraph induced by $U_l$ can extend upward to reach $U_j$ and downward to
reach $U_1$, at least one of which is connected.

Hence we may assume that $H'[U_l]-S$ is connected for each $l$ with
$i\le l\le j$.  For $i\le l<j$, the subgraph induced by $U_l\cup U_{l+1}$ is
also connected unless $S$ consists of all $k-1$ edges joining $U_l$ and 
$U_{l+1}$.  If $S$ is not any of these sets, then altogether $H'[U_l]-S$
is connected.  However, if $S$ consists of the $k-1$ edges joining 
$U_l$ and $U_{l+1}$, then the subgraph induced by $\VECOP Uil\cup$
and the subgraph induced by $\VECOP U{l+1}j\cup$ are connected, and the 
presence of $xy$ connects these two subgraphs.
\end{proof}

By Proposition~\ref{construction}, $\sat(n,\cF'_k) \le \rho_k(n)$.  Thus
$\sat(n,\cF'_k)$ is much smaller than $\sat(n,\cF_k)$ when $n\ge 2(k+1)$.
Indeed, $G_{k,n}$ is not $\cF_k$-saturated.  In particular, adding
an edge joining $u_{1,1}$ to a vertex $v$ outside $U_1$ does not create a
$k$-connected subgraph.  Since $G_{k,n}$ has no $k$-edge-connected subgraph, it
has no $k$-connected subgraph, so a $k$-connected subgraph $H'$ of the new
graph $G'$ must contain the edge $u_{1,1}v$.  Let $S=U_1-\{u_{1,2},u_{1,k}\}$;
note that $|S|=k-1$.  Since $H'$ must have $k-1$ internally disjoint paths from
$v$ to $u_{1,1}$ in addition to the edge $vu_{1,1}$, and $S$ is the set of
vertices in $U_1$ with neighbors outside $U_1$, all of $S$ must also lie in
$V(H')$.  Since $d_G(u_{1,k+1})=k$, we must also include $u_{1,2}$ and
$u_{1,k}$ in $V(H')$.  Now $H'-S$ has $u_{1,2}u_{1,k}$ as an isolated edge.

\section{Saturation and extremal number of $\mathcal{F}_k'$}

In this section, we show that if $G$ is an $\mathcal{F}_k'$-saturated
$n$-vertex graph with $n\ge k+1$, then $|E(G)| \ge \rho_k(n)$.
First, we investigate the properties of an $\mathcal{F}_k'$-saturated graph.

\begin{lem}\label{connectivity}
If $G$ is $\cF'_k$-saturated and has more than $k$ vertices, then
$\kappa'(G)=k-1$.
\end{lem}
\begin{proof}
Since $G$ has no $k$-edge-connected subgraph, $\kappa'(G)\le k-1$.
If $\kappa'(G)<k-1$, then $G$ has an edge cut $[S,\Sb]$ of size less than
$k-1$.  Since $|V(G)|>k$, there are at least $k$ pairs $(x,y)$ with $x\in S$
and $y\in \Sb$.  Hence there is such a pair $(x,y)$ with $xy\notin E(G)$.
Let $G'$ be the graph obtained by adding the edge $xy$ to $G$.

Since $G$ has no $k$-edge-connected subgraph, any such subgraph of $G'$
must contain the edge $xy$.  Hence it contains $k$ edge-disjoint paths
with endpoints $x$ and $y$, by Menger's Theorem.  Besides the edge $xy$, there
must be at least $k-1$ with endpoints $x$ and $y$ that use edges of $[S,\Sb]$.
This contradicts $|[S,\Sb]|<k-1$.  Hence $G'$ has no $k$-edge-connected
subgraph, and $G$ cannot be $\cF'_k$-saturated.
\end{proof}


\begin{lem}\label{mainlemma}
Assume $k\ge3$, and let $G$ be a $\cF'_k$-saturated graph with at least $k+2$
vertices.  If $S$ is a vertex subset in $V(G)$ such that $|[S,\Sb]|=k-1$ and
$|S|\ge|\Sb|$, then $G[S]$ is a nontrivial $\cF'_k$-saturated graph,
and $G[\Sb]$ is $K_1$ or is a nontrivial $\cF'_k$-saturated graph.
\end{lem}

\begin{proof}
First, we prove for $T\in\{S,\Sb\}$ that the induced subgraph $G[T]$ is a
complete subgraph or is $\cF'_k$-saturated with at least $k+1$ vertices.
When $G[T]$ is not complete, take $e\in E(\overline{G[S]})$, and let $G'$
be the graph obtained from $G$ by adding $e$.  Since $G$ is $\cF'_k$-saturated,
$G'$ contains a $k$-edge-connected subgraph $H$, and $e\in E(H)$.  Since
$|[T,\Tb]|=k-1$, no vertex of $H$ lies in $\Tb$.  Hence $H\esub G[T]$, which
implies that $G[T]$ is $\cF'_k$-saturated.  Since $G[T]$ is not complete, that
requires $|T|\ge k+1$.

If $G[\Sb]$ is a nontrivial $\cF'_k$-saturated graph, then $G[S]$ is also, by
$|S|\ge|\Sb|$ and the preceding paragraph.  If $G[\Sb]=K_1$, then
$|V(G)|\ge k+2$ and the preceding paragraph yield again that $G[S]$ is a
nontrivial $\cF'_k$-saturated graph.  Hence it suffices to show that $G[S]$
cannot be a complete graph with $|\Sb|\ge2$.

By Lemma~\ref{connectivity}, $\delta(G)\ge k-1$.  The vertex of $\Sb$ incident
to the fewest edges of $[S,\Sb]$ has degree at most $\FL{\FR{k-1}{j}}+j-1$,
where $j=|\Sb|$.  Since $j\ge2$, we thus have $j\ge k-1$.

If $j=k-1$, then $\delta(G)\ge k-1$ requires each vertex of $\Sb$ to be
incident to exactly one edge of the cut.  Adding an edge across the cut then
increases the degree of only one vertex of $\Sb$ to $k$.  Hence only that
vertex can lie in $H$, which restricts its degree in $H$ to $1$.

We may therefore assume $|\Sb|=k$, since $K_{k+1}\not\esub G$, and $|S|\ge k$.
Since $|[S,\Sb]|=k-1$, some $v\in \Sb$ has degree only $k-1$ in $G$, and every
vertex of $\Sb$ has a nonneighbor in $S$.  Choose $y\in \Sb$ with $y\ne v$, and
choose $x\in S$ with $xy\notin E(G)$.  Let $G'$ be the graph obtained by adding
$xy$ to $G$.  A $k$-edge-connected subgraph $H$ of $G'$ must contain $y$ but
cannot contain $v$.  If $H$ has $i+1$ vertices in $\Sb-\{v\}$, then a vertex
among these with least degree in $H$ has degree at most $\FL{\FR{k}{i+1}}+i$ in
$H$.  Since $i\le k-2$ and $\delta(H)\ge k$, we have $i=0$.

Hence $V(H)\cap\Sb=\{y\}$ and all edges of $[S,\Sb]$ are incident to $y$.
All vertices of $\Sb$ other than $y$ have degree $k-1$ in $G$.  In this
case let $G'$ be the graph obtained by adding $xv$ to $G$.  Since vertices
in the resulting $k$-edge-connected subgraph $H$ must have degree at least $k$,
the only vertices from $\Sb$ that can be included are $y$ and $v$.  However,
now $d_H(v)=2$, which prohibits such a subgraph $H$ since $k\ge3$.
\end{proof}

\begin{lem}\label{subgraph}
If $G$ is an $n$-vertex $\cF'_k$-saturated graph with $n\ge k+1$, then $G$
contains $K_{k+1}^-$.
\end{lem}
\begin{proof}
We use induction on $n$, the number of vertices.  The claim holds when $n=k+1$,
since $K_{k+1}^-$ is the only $\cF'_k$-saturated graph with $k+1$-vertices.

Now consider $n\geq k+2$.  Since $\kappa'(G)=k-1$ by Lemma \ref{connectivity},
there exists $S\subseteq V(G)$ such that $|[S,\Sb]|=k-1$ and $|S|\ge |\Sb|$.
By Lemma~\ref{mainlemma}, $|S|\ge k+1$ and $G[S]$ is $\mathcal{F}'_k$-saturated.
By the induction hypothesis, $G[S]$ (and hence also $G$) contains $K_{k+1}^-$.
\end{proof}

The lemmas allow us to prove the main result of this section.

\begin{thm}\label{main}
For $n\in\NN$, with $t=\FL{\FR n{k+1}}$,
$$\sat(n,\cF'_k)=(k-1)(n-1)-t\CH{k-1}2,$$
with equality achieved for $k=1$ by $\Kb_n$, for $k=2$ by trees, and for
$k\ge3$ by $G_{k,n}$.
\end{thm}
\begin{proof}
The claims for $k\le2$ are immediate.  For $k\ge3$,
Proposition~\ref{construction} yields the upper bound.
For the lower bound, we use induction on $n$.
When $n=k+1$, so $t=1$, the only $\cF'_k$-saturated $n$-vertex graph is
$K_{k+1}^-$, which indeed has $(k-1)k-\CH{k-1}2$ edges.

For $n>k+1$, let $G$ be a $\cF'_k$-saturated $n$-vertex graph.  Since
$\kappa'(G)=k-1$ by Lemma \ref{connectivity}, there exists $S\subseteq V(G)$
such that $|[S,\Sb]|=k-1$ and $|S|\ge |\Sb|$. By Lemma~\ref{mainlemma},
$G[S]$ is a nontrivial $\cF'_k$-saturated graph and $G[\overline{S}]$ is $K_1$
or is a nontrivial $\cF'_k$-saturated graph.  Let $t'=\FL{\FR{|S|}{k+1}}$.
By the induction hypothesis, $|E(G[S])|\ge (k-1)(|S|-1)-t'\CH{k-1}2$.

If $G[\Sb]=K_1$, then $|S|=n-1$ and exactly $k-1$ edges lie outside $G[S]$.
Hence $|E(G)\ge (k-1)(n-1)-t'\CH{k-1}2$.  Since $t'\in\{t,t-1\}$, the desired
inequality holds.

Therefore, we may assume that $G[S]$ and $G[\Sb]$ are both nontrivial
$\cF'_k$-saturated graphs.  Let $t''=\FL{\FR{|\Sb|}{k+1}}$.
Note that $t'+t''\le t$.  Using the induction hypothesis and adding the $k-1$
edges of the cut,
$$
|E(G)|\ge (k-1)(|S|+|\Sb|-2)+(k-1)-(t'+t'')\CH{k-1}2
\ge (k-1)(n-1)-t\CH{k-1}2.
$$

Hence $|E(G)|\ge \rho_k(n)$. 
\end{proof}

Next we determine the maximum number of edges in
$\cF'_k$-saturated $n$-vertex graphs.

\begin{thm}\label{main2}
If $n\ge k+1$, then $\ex(n,\mathcal{F}'_k)=(k-1)n-{k \choose 2}$.
Furthermore, the $n$-vertex $\cF'_k$-saturated graphs with the
most edges arise from $(n-1)$-vertex $\cF'_k$-saturated graphs with the
most edges by adding one vertex with $k-1$ neighbors.
\end{thm}

\begin{proof}
As we have noted, $\cF'_1$-saturated graphs have no edges and 
$\cF'_2$-saturated graphs are trees, so we may assume $k\ge3$.
We use induction on $n$; when $n=k+1$, the only $\cF'_k$-saturated $n$-vertex
graph is $K_{k+1}^-$.

For $n>k+1$, let $G$ be an $\cF'_k$-saturated $n$-vertex graph.
As in Theorem~\ref{main}, there exists $S\subseteq V(G)$
such that $|[S,\Sb]|=k-1$ and $|S|\ge |\Sb|$. By Lemma~\ref{mainlemma},
$G[S]$ is a nontrivial $\cF'_k$-saturated graph and $G[\overline{S}]$ is $K_1$
or is a nontrivial $\cF'_k$-saturated graph.

Applying the induction hypothesis, if $G[\Sb]=K_1$, then
$|E(G)|\le (k-1)(n-1)+(k-1)-\CH k2=(k-1)n-\CH k2$, with equality only under
the claimed condition.  On the other hand, if $[\Sb]$ is a nontrivial
$\cF'_k$-saturated graph, then 
$$|E(G)|\le (k-1)|S|-\CH k2+(k-1)|\Sb|-\CH k2+(k-1)
=(k-1)n-(k+1)(k-1).$$
Since $k+1 > k/2$ when $k>0$, the upper bound in this case is strictly
smaller than the claimed upper bound.
\end{proof}

\section{Spectral radius and $\mathcal{F}_k'$-saturated graphs}

In this section, we give sharp lower bounds on the spectral radius for
$\mathcal{F}_k'$-saturated graphs and for $\cF_k$-saturated graphs.
The spectral radius of a graph $G$, denoted $\lambda_1(G)$, is the largest
eigenvalue of the adjacency matrix of $G$.  The following two lemmas are
well-known in spectral graph theory.

\begin{lem}[\cite{GR2013}]\label{PF}
If $H$ is a subgraph of $G$, then $\lambda_1(H)\leq\lambda_1(G)$.
\end{lem}

\begin{lem}[\cite{BH2012}]\label{In}
For any graph $G$, $$\frac{2|E(G)|}{|V(G)|}\leq\lambda_1(G)\leq\Delta(G)$$ with equality if and only if $G$ is regular.
\end{lem}

For a vertex partition $P$ of a graph $G$, with parts $\VEC V1t$,
the \dfn{quotient matrix} $Q$ has $(i,j)$-entry $\frac{|[V_i,V_j]|}{|V_i|}$
when $i\neq j$ and $\frac{2|E(G[V_i])|}{|V_i|}$ when $i=j$.  Let $q_{i,j}$
denote the $(i,j)$-entry in $Q$.  A vertex partition $P$ with $t$ parts is
\dfn{equitable} if whenever $i,j\in[t]$ and $v\in V_i$, the number of 
neighbors of $v$ in $V_j$ is $q_{i,j}$.

\begin{lem}[\cite{GR2013}]\label{quotient}
If $\{V_1,\ldots, V_t\}$ is an equitable partition of $V(G)$, then
$\lambda_1(G)=\lambda_1(Q)$, where $Q$ is the quotient matrix for the partition.
\end{lem}

\begin{thm}\label{largest}
If $G$ is a nontrivial $\mathcal{F}'_k$-saturated graph, then
$\lambda_1(G)\geq(k-2+\sqrt{k^2+4k-4})/2$, with equality for $K_{k+1}^-$.
\end{thm}
\begin{proof}
First we prove
$\lambda_1(K_{k+1}^-)=(k-2+\sqrt{k^2+4k-4})/2$.  Let
$V(K_{k+1}^-)=\{x_1,\ldots,x_{k+1}\}$, with $d(x_1)=d(x_{k+1})=k-1$.
The vertex partition of $K_{k+1}^-$ given by $V_1=\{x_1,x_{k+1}\}$ and
$V_2=\{x_2,\ldots,x_k\}$ is equitable.  The corresponding quotient matrix $Q$
is $\CH{0\quad~~~2}{k-1~~k-2}$.
By Lemma \ref{quotient},
$\lambda_1(K_{k+1}^-)=\lambda_1(Q)=(k-2+\sqrt{k^2+4k-4})/2$.

For any nontrivial
$\mathcal{F}'_k$-saturated graph $G$,  Lemma \ref{subgraph} yields
$K_{k+1}^-\esub G$.  By Lemma~\ref{PF}, $\lambda_1(G)\geq\lambda_1(K_{k+1}^-)$,
as desired.
\end{proof}

\begin{thm}\label{lambda1}
If $G$ is $\mathcal{F}_k$-saturated with $n$ vertices, where $n\ge k+1$, then\\
$\lambda_1(G)\geq(k-2+\sqrt{k^2+4k-4})/2$.
\end{thm}
\begin{proof}
For $k=1$, the bound is $0$ and the eigenvalues have sum $0$,
so we may assume $k\ge2$.  When $n=k+1$, the only $\cF_k$-saturated graph is
$K_{k+1}^-$, whose spectral radius as computed in Theorem~\ref{largest} is the
claimed bound.  Hence we may assume $n\ge k+2\ge 4$.

By Theorem \ref{k-con}, $|E(G)|\geq(k-1)n-{k\choose2}$.
By Lemma \ref{In},
$$\lambda_1(G)\ge\FR{2|E(G)|}{n}\ge\FR{2(k-1)n-2\CH k2}{n}
=2(k-1)-\FR{k(k-1)}{n}.$$
Thus it suffices to prove $2(k-1)-k(k-1)/n\ge (k-2+\sqrt{k^2+4k-4})/2$.

For $k=2$, this reduces to $2-2/n\ge \sqrt{2}$, which holds when $n\ge4$.
For $k=3$, it reduces to $4-6/n\ge (1+\sqrt{17})/2$, which holds when $n\ge5$.

For $k\geq4$, since $k>(k-2+\sqrt{k^2+4k-4})/2$, it suffices to prove
$2(k-1)-\frac{k(k-1)}{n}\geq k$.  We compute
$$2(k-1)-\frac{k(k-1)}{n}-k\geq k-2-\frac{k(k-1)}{k+2}=\frac{k-4}{k+2}\geq0.$$

This completes the proof.
\end{proof}

For $t \ge 3$, let $\mathcal{F}_{d,t}$ be the family of $d$-regular simple
graphs $H$ with $\kappa'(H)\le t$.  In~\cite{HOPPY}, it was proved that
the minimum of the second largest eigenvalue over graphs in $\mathcal{F}_{d,t}$
is the second largest eigenvalue of a smallest graph in $\mathcal{F}_{d,t}$.
Theorem~\ref{largest} and~\ref{lambda1} similarly say that
the minima of the spectral radius over $\mathcal{F}$-saturated graphs and
over $\mathcal{F}'$-saturated graphs are the spectral radii of the smallest
graph in these families.


\end{document}